\documentclass[11pt]{article}

\usepackage{amsmath,amsfonts,amssymb}
\usepackage{amsthm}
\usepackage{color,ulem}

\renewcommand{\P}{\mathsf P}

\renewcommand{\a}{\alpha}

\begin{document}

\newtheorem{proposition}{Proposition}
\newtheorem{theorem}{Theorem}
\newtheorem{corollary}{Corollary}
\newtheorem{lemma}{Lemma}

\theoremstyle{definition}
\newtheorem{remark}{Remark}
\newtheorem{example}{Example}

\newcommand{\iid}{{\stackrel{\rm iid}{\sim}}}

\allowdisplaybreaks

\baselineskip3.5ex

\def\arraystretch{1.5}

\title{On the method of pivoting the CDF for exact \\ confidence intervals with illustration for exponential \\ mean under life-test with time constraints}

\author{N.~Balakrishnan$^1$ \and E.~Cramer$^2$ \and G.~Iliopoulos$^3$}

\date{}

\maketitle

\footnotetext[1]{Department of Mathematics and Statistics, McMaster University, Hamilton, Ontario, Canada L8S 4K1; e-mail: {\tt bala@mcmaster.ca}}
\footnotetext[2]{Institute of Statistics, RWTH Aachen University; 52056 Aachen, Germany; e-mail: {\tt erhard.cramer@rwth-aachen.de}}
\footnotetext[3]{(Corresponding author) Department of Statistics and Insurance Science, University of Piraeus; 80 Karaoli \& Dimitriou str., 18534 Piraeus, Greece; e-mail: {\tt geh@unipi.gr}}

\maketitle

\begin{abstract}
Two requirements for pivoting a cumulative distribution function (CDF) in order to construct exact confidence intervals or bounds for a real-valued parameter $\theta$ are the monotonicity of this CDF with respect to $\theta$ and the existence of solutions of some pertinent equations for $\theta$. The second requirement is not fulfilled by the CDF of the maximum likelihood estimator of the exponential scale parameter when the data come from some life-testing scenarios such as type-I censoring, hybrid type-I censoring, and progressive type-I censoring that are subject to time constraints. However, the method has been used in these cases  probably because the non-existence of the solution usually happens only with small probability. Here, we illustrate the problem by giving formal details in the case of type-I censoring and by providing some further examples. We also present a suitable extension of the basic pivoting method which is applicable in situations wherein the considered equations have no solution.
\end{abstract}


\section{Introduction}

Pivoting the CDF is a standard method for the construction of exact confidence intervals for a real parameter $\theta$ which dates back to 1934 when Clopper and Pearson used it to derive an exact confidence interval for the success probability of binomial distribution. The method was strictly formulated by Barlow et al.~(1968) and is well described in Casella and Berger (2002, Sec.~9.2.3). It is based on the inversion with respect to $\theta$ of the cumulative distribution function (CDF) $F(y;\theta)$ of a (typically) sufficient statistic $Y$. In the case when $Y$ is continuous, the method is applicable under the following conditions:
\begin{enumerate}
\item[I.] The distribution of $Y$ is stochastically increasing (resp., decreasing) in $\theta$, i.e., for all $y$, we have $F(y;\theta_1)\geqslant\mbox{(resp., $\leqslant$)}\, F(y;\theta_2)$ whenever $\theta_1<\theta_2$;
\item[II.] For some $\alpha_1$, $\alpha_2>0$ with $\alpha_1+\alpha_2=\alpha\in(0,1)$, the equations $F(y;\theta)=1-\alpha_1$ and $F(y;\theta)=\alpha_2$ can be solved with respect to $\theta$ for all possible values $y$ of $Y$.
\end{enumerate}
If we denote by $\theta_L(y)<\theta_U(y)$ the solutions to the above equations, the random interval $[\theta_L(Y),\theta_U(Y)]$ is then an exact $100(1-\alpha)\%$ confidence interval for $\theta$. The typical choice is $\alpha_1=\alpha_2=\alpha/2$, though it is not necessary in general. By slightly generalizing condition II, the method can be also applied in cases when the distribution of $Y$ is noncontinuous (e.g.,~discrete); see Casella and Berger (2002).

In the last 45 years, the above approach has been adopted by many researchers for constructing exact confidence intervals and/or bounds for the mean of the exponential distribution based on data obtained from 
life-testing experiments under time constraints. Schemes such as hybrid type-I censoring, generalized hybrid type-II censoring and progressive type-I censoring  can be thought of as extensions of the conventional type-I (right) censoring, and were introduced in the literature to overcome  some  drawbacks in the conventional type-I censoring. Recall that a random sample is said to be type-I right censored, hereafter simply ``type-I censored'', if only observations up to a certain time $T$ are observed. To fix ideas, suppose $n$ items are subject to a life-test and denote by $X_1,\ldots,X_n$ their lifetimes. Assume that they are iid random variables from an exponential distribution with mean $\theta\in\Theta=(0,\infty)$. Let $T>0$ be a fixed time point and suppose the life-test is terminated at time $T$. If $D=\sum_{i=1}^n I(X_i\leqslant T)$ denotes the number of failures observed 
 until time $T$, then the maximum likelihood estimator (MLE) of $\theta$ exists if and only if $D\geqslant 1$ and is given by
\[
\hat{\theta} = \frac{\sum_{i=1}^D X_{i:n}+(n-D)T}{D},
\]
where $X_{1:n}<\cdots<X_{n:n}$ is the ordered sample (cf.~Arnold et al., 2008). But when $D=0$, the MLE does not exist as in this case it can be readily seen that the likelihood function is a monotone 
increasing function of $\theta$. The estimator $\hat{\theta}$ seems to have been proposed for the first time by Bartlett (1953). Its conditional CDF, given $D\geqslant 1$, was derived in closed-form by Bartholomew (1963). It can be represented as a mixture of left-truncated gamma distributions with both positive and negative weights (sometimes referred to as a generalized mixture) with the shape parameters of these gamma distributions being integers ranging from $1$ to $n$. The exact form of this CDF,  denoted by $F(y;\theta|D\geqslant 1)$, is presented in Section \ref{type-I censoring}.

Since $F(y;\theta|D\geqslant 1)$ is available in closed-form, it seems natural to pivot it in order to construct exact confidence intervals for $\theta$. In fact, this particular approach has been suggested by Barlow et al.~(1968) who also provided a computer program to calculate the corresponding confidence limits. However, although in this case condition I holds (see Balakrishnan et al., 2002, or Balakrishnan and Iliopoulos, 2009), condition II fails. Indeed, in Section 3 (see Corollary \ref{range of F}), it is shown that for any $u\in(0,1)$, the range of $F((n-1+u)T;\theta|D\geqslant 1)$ as a function of $\theta$ is the interval $(u,1)$. This implies that whatever $\alpha_1$ and $\alpha_2$  we choose, there is always a positive probability set $B$ such that for any $y\in B$ at least one of the equations  in condition II has no solution. Hence, the method of pivoting the CDF can not be applied in this case unless it is modified suitably. It turns out that the same problem appears for the CDF of the MLE of $\theta$ for many other life-testing scenarios such as those mentioned earlier.  However, this problem does not seem to have been noted while the method itself has been adopted in several sampling scenarios.

The rest of this paper is organized as follows. In Section \ref{the extension}, we present an extension of the method of pivoting the CDF that covers situations like those mentioned above. We prove that the confidence intervals constructed by this extension still remain exact. In Section \ref{type-I censoring}, we discuss in detail the case of type-I right censored exponential lifetimes and prove that the original pivoting method can never be applied as it is. In Section \ref{other sampling schemes}, we point out some more general life-testing experiments wherein the same problem persists.  We conclude with a brief discussion in Section \ref{discussion}.

\section{An extension of the method of pivoting the CDF}
\label{the extension}

Let $Y$ be a statistic with CDF $F(y;\theta)$ which is continuous in both $y\in\mathbb R$ and $\theta\in\Theta$. Assume that $\Theta$ is an open, half-open or closed interval with endpoints $\underline{\theta}<\overline{\theta}$ which may be $-\infty$ and/or $\infty$, respectively. Assume further that $Y$ is stochastically increasing in $\theta$, i.e., $F(y;\theta)$ is a decreasing function of $\theta$ for each $y$ and let
\[
F(y;\underline{\theta}) = \lim_{\theta\downarrow\underline{\theta}}F(y;\theta), \quad F(y;\overline{\theta}) = \lim_{\theta\uparrow\overline{\theta}}F(y;\theta).
\]
By the monotonicity of $F(y;\theta)$ with respect to $\theta$ and the fact that it is bounded, these limits exist for all $y$. To keep things simple, let us assume that the monotonicity with respect to $\theta$ is strict. Then, for any $y\in\mathcal Y={\rm support}(Y)$ and $\alpha\in(0,1)$, the equation $F(y,\theta)=\alpha$ will have a solution $\theta(\alpha,y)\in(\underline{\theta},\overline{\theta})$ if and only if $F(y;\overline{\theta})<\alpha<F(y;\underline{\theta})$. Define
\begin{equation}\label{theta star}
\theta^*(\alpha,y) = \left\{\begin{array}{ll} \underline{\theta},& \mbox{if $F(y;\underline{\theta})\leqslant\alpha$}, \\
\theta(\alpha,y),& \mbox{if $F(y;\overline{\theta})<\alpha<F(y;\underline{\theta})$}, \\ \overline{\theta},& \mbox{if $F(y;\overline{\theta})\geqslant\alpha$}. \end{array}\right.
\end{equation}
Then, we have the following result.

\begin{lemma}\label{P theta star}
For any $\alpha\in(0,1)$, $\P_{\theta}\{\theta^*(\alpha,Y)\leqslant\theta\}=\alpha$.
\end{lemma}
\begin{proof}
Partition $\mathcal Y$ into  non-overlapping sets $B_1(\alpha)=\{y:\theta^*(\alpha,y)=\underline{\theta}\}$, $B_2(\alpha)=\{y:\theta^*(\alpha,y)=\theta(\alpha,y)\}$, $B_3(\alpha)=\{y:\theta^*(\alpha,y)=\overline{\theta}\}$, and 
then set $y_1=\sup B_1(\alpha)$, $y_2=\sup B_2(\alpha)$. Note that $y_1\leqslant \inf B_2(\alpha)$ and $y_2\leqslant\inf B_3(\alpha)$. Indeed, if $y\in B_1(\alpha)$ and $y'<y$, then $F(y';\theta)\leqslant F(y;\theta)$ and by taking the limits as $\theta\downarrow\underline{\theta}$, we see that $F(y';\underline{\theta})\leqslant\alpha$ as well. Similarly, if $y\in B_3(\alpha)$ and $y'>y$, then $y'\in B_3(\alpha)$, too. Hence, $B_1(\alpha)$ contains ``small'' $y$'s and $B_3(\alpha)$ contains ``large'' $y$'s, and so $B_2(\alpha)$ lies between these two sets. It turns out that $\P_{\theta}\{Y\in B_1(\alpha)\}=F(y_1;\theta)$, $\P_{\theta}\{Y\in B_2(\alpha)\}=F(y_2;\theta)-F(y_1;\theta)$, and $\P_{\theta}\{Y\in B_3(\alpha)\}=1-F(y_2;\theta)$. Applying now the theorem of total probability, we get for any $\theta\in(\underline{\theta},\overline{\theta})$,
\begin{align}
\P_{\theta}\{\theta^* & (\alpha,Y)\leqslant\theta\} = \sum_{i=1}^{3}\P_{\theta}\{\theta^*(\alpha,Y)\leqslant\theta|Y\in B_i(\alpha)\}\P_{\theta}\{Y\in B_i(\alpha)\} \nonumber
\\
 =\,& \P_{\theta}\{Y\in B_1(\alpha)\} + \P_{\theta}\{\theta^*(\alpha,Y)\leqslant\theta|Y\in B_2(\alpha)\}\P_{\theta}\{Y\in B_2(\alpha)\} \nonumber
\\
=\,& F(y_1;\theta) + \P_{\theta}\{\theta^*(\alpha,Y)\leqslant\theta|Y\in B_2(\alpha)\}\{F(y_2;\theta)-F(y_1;\theta)\}. \label{line3}
\end{align}
By definition, for $y\in B_2(\alpha)$, we have $\theta(\alpha,y)\leqslant\theta\Leftrightarrow F(y;\theta)\leqslant F(y;\theta(\alpha,y))=\alpha$. Notice that the conditional CDF of $Y$, given $Y\in B_2(\alpha)$, is the truncated continuous CDF
\[
F_2(y;\theta)=\frac{F(y;\theta)-F(y_1;\theta)}{F(y_2;\theta)-F(y_1;\theta)},\quad y\in B_2(\alpha).
\]
Thus,
\begin{align*}
& \P_{\theta}\{\theta^*(\alpha,Y)\leqslant\theta|Y\in B_2(\alpha)\} = \P_{\theta}\{F(Y;\theta)\leqslant\alpha|Y\in B_2(\alpha)\} = \\ & \P_{\theta}\bigg\{F_2(Y;\theta)\leqslant \frac{\alpha-F(y_1;\theta)}{F(y_2;\theta)-F(y_1;\theta)}\,\Big|\,Y\in B_2(\alpha)\bigg\} = \frac{\alpha-F(y_1;\theta)}{F(y_2;\theta)-F(y_1;\theta)},
\end{align*}
since, conditional on $Y\in B_2(\alpha)$, $F_2(Y;\theta)$ follows a uniform$(0,1)$ distribution. Upon substituting the last quantity in \eqref{line3}, we get the required result.
\end{proof}

\begin{theorem}\label{thm I star}
Let $\alpha_1,\alpha_2>0$ with $\alpha_1+\alpha_2=\alpha\in(0,1)$. Then, the random set
\begin{equation*}\label{I star}
I^*_{\alpha_1,\alpha_2}(Y)=[\theta^*(1-\alpha_1,Y),\theta^*(\alpha_2,Y)]
\end{equation*}
is an exact $100(1-\alpha)\%$ confidence interval for $\theta$.
\end{theorem}
\begin{proof}
Note first that the monotonicity of $F(y;\theta)$ with respect to $\theta$, together with the fact that $1-\alpha_1>\alpha_2$, imply $\theta^*(1-\alpha_1,y)\leqslant\theta^*(\alpha_2,y)$ for all $y$. By writing $\P_{\theta}\{\theta^*(1-\alpha_1,Y)\leqslant\theta\leqslant\theta^*(\alpha_2,Y)\}=\P_{\theta}\{\theta^*(1-\alpha_1,Y)\leqslant\theta\}-\P_{\theta}\{\theta^*(\alpha_2,Y)\leqslant\theta\}$ and then applying Lemma \ref{P theta star},  the result follows.
\end{proof}

When $\theta(1-\alpha_1,y)$ and $\theta(\alpha_2,y)$ are defined for all values $y$, the confidence interval $I^*_{\alpha_1,\alpha_2}(Y)$ is identical to the one given in Casella and Berger (2002) since in this case the original method works properly. However, if one of these quantities does not exist for some $y$, then $I^*_{\alpha_1,\alpha_2}(y)$ becomes simply a bound (either lower or upper) for $\theta$. It is indeed inconvenient when $\underline{\theta}=-\infty$ or $\overline{\theta}=\infty$ since $\theta^*(1-\alpha_1,Y)=-\infty$ or $\theta^*(\alpha_2,Y)=\infty$ with positive probability implies that the expected width of $I^*_{\alpha_1,\alpha_2}(Y)$ is infinite. Moreover, there may be values $y$ such that $I^*_{\alpha_1,\alpha_2}(y)$ degenerates into the ``points'' $\underline{\theta}$ or $\overline{\theta}$ which may lie outside the parameter space. This follows from the facts that $\theta^*(1-\alpha_1,y)=\overline{\theta}\Rightarrow\theta^*(\alpha_2,y)=\overline{\theta}$ and $\theta^*(\alpha_2,y)=\underline{\theta}\Rightarrow\theta^*(1-\alpha_1,y)=\underline{\theta}$. Thus, such values of $Y$ may be useless for the construction of this particular confidence interval. However, exclusion of these values from the analysis by truncating $Y$ off the ``useless'' set inflates the actual coverage probability to $(1-\alpha)/\P_\theta\{Y\in B(\alpha_1,\alpha_2)\}$, where $B(\alpha_1,\alpha_2)=\{y;\mbox{$\theta^*(1-\alpha_1,y)<\overline{\theta}$ and $\theta^*(\alpha_2,y)>\underline{\theta}$}\}$. This would result in 
 the corresponding confidence interval being quite conservative.

Situations where $F(y;\theta)=\alpha$ has no solution are not artificial at all as we shall see in the next section. It could occur in any standard model under restriction on the parameter space. For instance, suppose we wish to construct an exact confidence interval for a normal mean $\theta$ based on a random sample $X_1,\ldots,X_n$, but subject to the restriction $\underline{\theta}\leqslant\theta\leqslant\overline{\theta}$, where $\underline{\theta}<\overline{\theta}$ are known finite values. When the parameter space is unrestricted, then pivoting the CDF of the sufficient statistic $\bar{X}$ is straightforward and it results in standard confidence intervals for $\theta$ of the form $[\bar{X}-z_{\a_2}\sigma/\sqrt{n},\bar{X}+z_{\alpha_1}\sigma/\sqrt{n}\,]$ (assuming that the standard deviation $\sigma$ is known for simplicity), where $z_\alpha$ denotes the upper $\alpha$-quantile of the standard normal distribution. This is a consequence of the fact that the CDF of $\bar{X}$, $\Phi(\sqrt{n}(y-\theta)/\sigma)$, satisfies $\lim_{\theta\downarrow-\infty}\Phi(\sqrt{n}(y-\theta)/\sigma)=1$, $\lim_{\theta\uparrow\infty}\Phi(\sqrt{n}(y-\theta)/\sigma)=0$ for any $y\in\mathbb R$. However, this is not the case when $\theta$ is restricted to the interval $[\underline{\theta},\overline{\theta}]$. Then, for a range of $y$ that has non-zero probability, the equations $\Phi(\sqrt{n}(y-\theta)/\sigma)=1-\alpha_1$ and/or $\Phi(\sqrt{n}(y-\theta)/\sigma)=\alpha_2$ have no solution. The modification we have proposed  forces for such $y$ the endpoints of the confidence intervals to be $\underline{\theta}$ and/or $\overline{\theta}$, i.e., it applies the natural truncation whenever needed.

\begin{remark}
In case where the monotonicity of $F(y;\theta)$ with regard to $\theta$ is not strict for some $y$,  the equation $F(y;\theta)=\alpha$ may have multiple solutions. Then, by setting $\theta(\alpha,y)=\inf\{\theta:F(y,\theta)=\alpha\}$, Lemma \ref{P theta star} and Theorem \ref{thm I star} continue to hold since the crucial equivalence $\theta(\alpha,y)\leqslant\theta\Leftrightarrow F(y;\theta)\leqslant F(y;\theta(\alpha,y))=\alpha$ still remains valid.
\end{remark}

\begin{remark}
In cases where $Y$ is stochastically decreasing in $\theta$, i.e.,~when $F(y;\theta)$ is increasing in $\theta$ for all $y$, the distribution can be reparametrized by $\theta'=\psi(\theta)$ where $\psi$ is a strictly decreasing function. Then, $Y$ is stochasticaly increasing in $\theta'$ and everything works as above.
\end{remark}

\begin{remark}\label{noncontinuous}
The extension of the method may be adapted to the case where $Y$ has a noncontinuous distribution. Following Casella and Berger (2002), let $\bar{F}(y;\theta)=\P_{\theta}(Y\geqslant y)$. For any $y\in\mathcal Y$, this function is increasing in $\theta$ since $F(y;\theta)$ is decreasing. For $\alpha\in(0,1)$ and $y\in\mathcal Y$, let us use $\theta'(\alpha,y)$ to denote the solution of the equation $\bar{F}(y;\theta)=\alpha$ with respect to $\theta$, when it  exists. Define
\begin{equation*}
\theta_*(\alpha,y) = \left\{\begin{array}{ll} \underline{\theta},& \mbox{if $\bar{F}(y;\underline{\theta})\geqslant\alpha$}, \\
\theta'(\alpha,y),& \mbox{if $\bar{F}(y;\underline{\theta})<\alpha<\bar{F}(y;\overline{\theta})$}, \\
\overline{\theta},& \mbox{if $\bar{F}(y;\overline{\theta})\leqslant\alpha$}. \end{array}\right.
\end{equation*}
Then, for any $\alpha_1,\alpha_2>0$ with $\alpha_1+\alpha_2=\alpha\in(0,1)$, the random set $[\theta_*(\alpha_1,Y),\theta^*(\alpha_2,Y)]$ is a $100(1-\alpha)\%$ confidence interval for $\theta$. This follows from the facts that $\theta_*(\alpha_1,y)\leqslant\theta^*(\alpha_2,y)$ for all $y$ as well as $\P_{\theta}\{\theta_*(\alpha_1,Y)>\theta\}<\alpha_1$ and $\P_{\theta}\{\theta^*(\alpha_2,Y)<\theta\}<\alpha_2$. The last inequalities can be established as done in Lemma \ref{P theta star}.
\end{remark}

\section{Type-I censoring with exponential lifetimes}
\label{type-I censoring}

Recall the type-I censoring setup with exponential lifetimes mentioned in the Introduction. By using a moment generating function approach as done by Bartholomew (1963), it can be shown that for any integer $1\leqslant d_0\leqslant n$ the conditional CDF of $\hat{\theta}$, given $D\geqslant d_0$, is
\[
F(y;\theta|D\geqslant d_0) = \sum_{d=d_0}^n\sum_{\nu=0}^d \frac{(-1)^\nu\binom{n}{d}\binom{d}{\nu}}{\P_\theta(D\geqslant d_0)}e^{-(n-d+\nu)T/\theta}G\bigg(\frac{dy-(n-d+\nu)T}{\theta}\,;\,d\bigg),
\]
where $G(x;d)=\Gamma(d)^{-1}\int_0^x u^{d-1}e^{-u}{\rm d}u$, $x>0$, denotes the CDF of the gamma distribution with shape $d$ and scale $1$. The original Bartholomew's expression was given (in a slightly different form) for $d_0=1$ since this is the minimum requirement for $\hat{\theta}$ to exist, but the reason for defining the conditional CDF for any $d_0$ will become apparent at the end of this section. Note that the support of the above distribution is the set $\cup_{d=d_0}^n [(n-d)T/d,nT/d]$ which in general is not connected (cf.~Cramer and Balakrishnan, 2013). However, the CDF $F(y;\theta|D\geqslant d_0)$ is continuous for all $y>0$ as well as $\theta>0$.

\begin{lemma}\label{lemma limit}
For any integers $d,d_0$ such that $1\leqslant d_0\leqslant d\leqslant n$,
\[
\lim_{\theta\uparrow\infty}\frac{G(\{dy-(n-d+\nu)T\}/\theta\,;\,d)}{\P_{\theta}(D\geqslant d_0)} = \left\{ \begin{array}{ll} 0, & d>d_0, \\ \dfrac{(n-d_0)!}{n!}\dfrac{\{d_0 y-(n-d_0+\nu)T\}_+^{d_0}}{T^{d_0}}, & d=d_0, \end{array}\right.
\]
where $(a)_+=\max\{a,0\}$.
\end{lemma}
\begin{proof}
Consider the Poisson sum representation of the CDF of the gamma distribution with integer shape $d$ and scale $1$, $G(u;d) = \sum_{j=d}^{\infty}e^{-u}u^j/j!=e^{-u}u^d/d!+o(u^d)$ as $u$ approaches zero. Clearly, if $dy-(n-d+\nu)T\leqslant 0$, the quantity on the left hand side equals zero and so its limit is zero as well. When $dy-(n-d+\nu)T>0$,
\begin{multline*}
\frac{G(\{dy-(n-d+\nu)T\}/\theta\,;\,d)}{\P_{\theta}(D\geqslant d_0)} = \frac{\sum_{j=d}^{\infty}e^{-\{dy-(n-d+\nu)T\}/\theta}\{dy-(n-d+\nu)T\}^j/(\theta^j j!)} {\sum_{j=d_0}^n\binom{n}{j}(1-e^{-T/\theta})^je^{-(n-j)T/\theta}}
\\ =
\frac{e^{-\{dy-(n-d+\nu)T\}/\theta}\{dy-(n-d+\nu)T\}^d/(\theta^d d!)+o(\theta^{-d})} {\binom{n}{d_0}(1-e^{-T/\theta})^{d_0}e^{-(n-d_0)T/\theta}+o(\theta^{-d_0})}
\end{multline*}
as $\theta\uparrow\infty$. Since $(1-e^{-T/\theta})^{d_0}=O(\theta^{-d_0})$, the limit equals zero provided $d>d_0$. On the other hand, if $d=d_0$,  it can be verified that the limit is as stated above.
\end{proof}

\begin{theorem}\label{limit F}
For any integer $d_0$ such that $1\leqslant d_0\leqslant n$, we have
\[
{\displaystyle\lim_{\theta\downarrow 0}}F(y;\theta|D\geqslant d_0)=1, \quad \forall\,y>0,
\]
and
\[
\lim_{\theta\uparrow\infty}F(y;\theta|D\geqslant d_0) = \left\{ \begin{array}{ll} 0, & y\leqslant (n-d_0)T/d_0, \\ \displaystyle\sum_{\nu=0}^{d_0-1}\dfrac{(-1)^\nu\{d_0 y-(n-d_0+\nu)T\}_+^{d_0}}{\nu!(d_0-\nu)!\,T^{d_0}}, & (n-d_0)T/d_0<y\leqslant nT/d_0. \end{array}\right.
\]
Moreover, the last sum takes every value in $(0,1)$ as $y$ ranges over $((n-d_0)T/d_0,nT/d_0)$.
\end{theorem}
\begin{proof}
Note that as $\theta\downarrow 0$, $1-e^{-T/\theta}\uparrow 1$ and so $D$ converges in probability to $n$. This implies that $\lim_{\theta\downarrow 0}\P_{\theta}(D\geqslant d_0)=1$ for all $d_0$. On the other hand, as $\theta\downarrow 0$, the only term that matters in the sum representation of $F(y;\theta|D\geqslant d_0)$ is the one corresponding to $d=n$, $\nu=0$, since in all other cases we have $n-d+\nu>0$ and so $\lim_{\theta\downarrow 0}e^{-(n-d+\nu)T/\theta}=0$. Hence, for all $y>0$,
\[
\lim_{\theta\downarrow 0}F(y;\theta|D\geqslant d_0) = \lim_{\theta\downarrow 0} \frac{G(ny/\theta\,;\,n)}{\P_\theta(D\geqslant d_0)} = 1,
\]
and thus the first limit is established. Observe now that by Lemma \ref{lemma limit} the ratios $G(\{dy-(n-d+\nu)T\}/\theta;d)/\P_{\theta}(D\geqslant d_0)$ converge to zero as $\theta\uparrow\infty$ for all $d>d_0$. Hence,
\[
\lim_{\theta\uparrow\infty}F(y;\theta|D\geqslant d_0) = \lim_{\theta\uparrow\infty}\sum_{\nu=0}^{d_0} \frac{(-1)^\nu\binom{n}{d_0}\binom{d_0}{\nu}}{\P_\theta(D\geqslant d_0)}e^{-(n-d_0+\nu)T/\theta}G\bigg(\frac{d_0y-(n-d_0+\nu)T}{\theta}\,;\,d_0\bigg).
\]
Note, however, that if $y\leqslant(n-d_0)T/d_0$, all of the CDFs appearing on the right hand side equal zero and so the result follows trivially. On the other hand, for $y\in((n-d_0)T/d_0,nT/d_0]$, another application of Lemma \ref{lemma limit} shows that the limit is as stated above. (The term corresponding to $\nu=d_0$  always equals zero.)

Finally,  to prove that the sum takes all values in $(0,1)$, observe first that it is continuous and strictly increasing in $y$. Its limit as $y\downarrow(n-d_0)T/d_0$ is obviously zero. On the other hand, as $y\uparrow nT/d_0$, the limit becomes
\[
\sum_{\nu=0}^{d_0-1}\dfrac{(-1)^\nu\{d_0(n T/d_0)-(n-d_0+\nu)T\}^{d_0}}{\nu!(d_0-\nu)!\,T^{d_0}} =
\sum_{\nu=0}^{d_0-1}\dfrac{(-1)^\nu(d_0-\nu)^{d_0}}{\nu!(d_0-\nu)!} = \sum_{k=1}^{d_0}\frac{(-1)^{d_0-k}k^{d_0}}{(d_0-k)!k!}.
\]
The last sum is the Stirling number of the second kind $S(d_0,d_0)$ (cf.~Charalambides, 2005) which is equal to $1$, and this completes the proof.
\end{proof}

\begin{corollary}\label{range of F}
For any $u\in(0,1)$, the range of $F((n-1+u)T;\theta|D\geqslant 1)$ as a function of $\theta$ is the interval $(u,1)$.
\end{corollary}
\begin{proof}
For any $u\in(0,1)$, we have  $(n-1+u)T\in((n-1)T,nT)$. Hence, by applying Theorem \ref{limit F} and using the continuity of $F((n-1+u)T;\theta|D\geqslant 1)$ with respect to $\theta$, we get that its range as a function of $\theta$ is the interval
\[
\big(\lim_{\theta\uparrow\infty}F((n-1+u)T;\theta|D\geqslant 1),\lim_{\theta\downarrow 0}F((n-1+u)T;\theta|D\geqslant 1)\big) = (u,1)
\]
as claimed.
\end{proof}

Corollary \ref{range of F} tells us that whatever $\alpha_1$ and $\alpha_2$  we choose for calculating the exact $100(1-\alpha)\%$ confidence interval for $\theta$, there is always a positive probability, no matter how small, where at least one of the endpoints is not defined. More specifically, $\theta(\alpha_2,y)$ is not defined for $y\geqslant(n-1+\alpha_2)T$, while both endpoints $\theta(1-\alpha_1,y)$ and $\theta(\alpha_2,y)$ are not defined for $y\geqslant(n-1+(1-\alpha_1))T$.

\begin{proposition}
For any $u\in(0,1)$, the probability that $\theta(u,\hat{\theta})$ is not defined lies between $0$ and $1-u$.
\end{proposition}
\begin{proof}
The quantity $\theta(u,y)$ is not defined when $y\geqslant(n-1+u)T$. This event has probability $1-F((n-1+u)T;\theta)$. By Corollary \ref{range of F}, $F((n-1+u)T;\theta)$ lies between $u$ and $1$, and hence the result.
\end{proof}

\begin{corollary}\label{pinfty pempty}
Let $\alpha_1,\alpha_2>0$ with $\alpha_1+\alpha_2=\alpha\in(0,1)$. Consider the exact $100(1-\alpha)\%$ confidence interval $I^*_{\alpha_1,\alpha_2}(\hat{\theta})$ for $\theta$. \\
(a) The probability $p_\infty$ that its upper point is infinite (so that the width of $I^*_{\alpha_1,\alpha_2}(\hat{\theta})$ is infinite) lies between $0$ and $1-\alpha_2$; \\
(b) The probability $p_\varnothing$ that both of its endpoints are infinite (so that $I^*_{\alpha_1,\alpha_2}(\hat{\theta})$ is an empty set) lies between $0$ and $\alpha_1$.
\end{corollary}

Notice that the problem of nonexistence of the solutions $\theta(\alpha,y)$ occurs only when $y\in((n-1)T,nT)$. This range of values are achieved by the MLE if and only if $D=1$. One may therefore think that a single observation is not enough for the construction of a  confidence interval with finite endpoints and that the problem can be solved by requiring more than one failure to be observed. However, whatever number of observations is required, the problem remains. Indeed, Theorem \ref{limit F} shows that no matter what $d_0$ is, there is always a range of values where the limit of $F(y;\theta|D\geqslant d_0)$ as $\theta\uparrow\infty$ may take any value between zero and one. Fortunately, $p_\infty$ and $p_\varnothing$ decrease to zero with $n$ at an exponential rate. Indeed, these probabilities are strictly smaller than $\P_\theta(D=1)/\P_\theta(D\geqslant 1)=n(e^{T/\theta}-1)/(e^{nT/\theta}-1)$, 
 and so most of the times $I^*_{\alpha_1,\alpha_2}(\hat{\theta})$ results in a proper interval. Nevertheless, $p_\infty$ and $p_\varnothing$ always remain positive which means that the expected width of $I^*_{\alpha_1,\alpha_2}(\hat{\theta})$ is infinite and there is also a positive probability to become an empty set.

\section{Some other scenarios facing the same problem}
\label{other sampling schemes}

As mentioned in the Introduction, Barlow et al.~(1968) applied the method of pivoting the CDF for constructing exact confidence limits for the exponential mean under type-I censoring. Since then, the approach has been used for the same purpose in more general models. 
In this section, we briefly review one of these models and demonstrate that for constructing  exact confidence intervals by pivoting the CDF of the corresponding MLE the extension presented in Section 2 must also be applied.

One of the flexible censoring schemes that appeared in the literature is hybrid type-I censoring. This scheme was originally proposed by Epstein (1954) and was further studied by Chen and Bhattacharyya (1988) and Childs et al.~(2003); see also Balakrishnan and Kundu (2013). Chen and Bhattacharyya (1988) and Childs et al.~(2003) pivot the CDF (more specifically, the survival function) of the MLE of $\theta$ in order to obtain exact confidence bounds and intervals, respectively. The difference of this scheme from the standard type-I censoring is that the experiment terminates at the random time $T^*=\min\{X_{r:n},T\}$, where $1\leqslant r\leqslant n$. The idea is that in cases wherein the experimenter is satisfied with $r$ observations, he/she is allowed to stop the experiment before the pre-determined time as soon as $r$ failures have occured. Assuming exponential lifetimes, Epstein (1954) found the MLE of $\theta$ in this case as
\[
\hat{\theta} = \frac{1}{D^*}\bigg\{\sum_{i=1}^{D^*} X_{i:n}+(n-D^*)T^*\bigg\},
\]
where $D^*=\sum_{i=1}^{n}I(X_i\leqslant T^*)$ denotes the number of failures that are observed up to time $T^*$, that is, $D^*=r$ if $T^*=X_{r:n}$ and $D^*=D$ if $T^*=T$, where $D$ is as defined before. As in the case of conventional type-I censoring, this MLE is also defined if and only if $D\geqslant 1$. Chen and Bhattacharyya (1988) gave an expression for the conditional CDF of $\hat{\theta}$ given $D\geqslant 1$, which was further simplified by Childs et al.~(2003). Their expression of the CDF is
\begin{align*}
F_r(y;\theta & |D\geqslant 1) = \bigg\{\sum_{d=1}^{r-1}\sum_{\nu=0}^{d}(-1)^\nu\binom{n}{d}\binom{d}{\nu}e^{-(n-d+\nu)T/\theta}G\bigg(\frac{dy-(n-d+\nu)T}{\theta}\,;\,d\bigg)
\\ &
\phantom{\bigg\{}+r\binom{n}{r}\sum_{\nu=1}^r\frac{(-1)^\nu e^{-(n-r+\nu)T/\theta}}{n-r+\nu}\binom{r-1}{\nu-1}G\bigg(\frac{ry-(n-r+\nu)T}{\theta}\,;\,r\bigg)
\\ &
\phantom{\bigg\{}+G(ry/\theta;r)\bigg\}\Big/\P_\theta(D\geqslant 1), \qquad\qquad 0<y<nT.
\end{align*}
Thus, in this case as well, the distribution is represented as a (generalized) mixture of gamma distributions. Its stochastic monotonicity has been established by Balakrishnan and Iliopoulos (2009). On the other hand, it is easy to verify that there is a range of values for which the limit of the CDF as $\theta\uparrow\infty$ is not zero. More specifically, for any $u\in(0,1)$, we have $\lim_{\theta\uparrow\infty}F_r((n-1+u)T;\theta|D\geqslant 1)=u$ when $r>1$ and $\lim_{\theta\uparrow\infty}F_r(nuT;\theta|D\geqslant 1)=u$ when $r=1$ which means that the problem of nonexistence of a solution to the equation $F_r(y;\theta|D\geqslant 1)=u$ for particular values $y$ is also present here. The same situation arises under type-I hybrid progressive censoring introduced by Childs et al.~(2008) (see also Cramer and Balakrishnan, 2013), generalized type-II hybrid censoring (Chandrasekar et al., 2004), progressive type-I censoring (Balakrishnan, 2007; Balakrishnan et al., 2011) and some other life-testing scenarios as well.

\section{Discussion}
\label{discussion}

The construction of exact confidence intervals for the scale parameter of the exponential distribution under life-test with time constraints has been discussed in a variety of censoring scenarios.  The method has then been 
  compared by simulation with some alternate  approaches like normal approximation, bootstrap confidence intervals and Bayesian credible intervals. Obviously, exact confidence intervals outperform confidence intervals based on such approximate methods in terms of coverage probability. In some cases, simulated average widths are reported as well. However, as illustrated in Corollary \ref{pinfty pempty} and the ensuing discussion, the expected width of the exact confidence interval is infinite. This is a consequence of the fact that for particular values $y$ of the MLE the equation $F(y;\theta)=\alpha$ has no solution, violating condition II given in the Introduction. It appears that this condition has not been noticed as a restriction and has been taken for granted. 
Normally, from time to time in Monte Carlo simulations, this problem must have led numerical methods to behave strangely. For example, the Newton method would keep iterating until overflow while the bisection method would not even be able to perform the initial step. In this paper, we have formally analyzed the situation in the case of conventional type-I censoring and discussed briefly other sampling schemes under which the CDF of the MLE behaves in a similar manner. 
The analysis performed in Section 3 suggests that there is no satisfactory solution to the problem of having both finite expected width and attainment of the coverage probability. For instance, truncation of the right endpoint from infinity to any finite value obviously decreases the coverage probability.

Running experiments subject to time constraints leads naturally to truncated distributions. One may wonder whether truncation of a distribution always does cause troubles to the existence of the solution $\theta(\alpha,y)$ of the equation $F(y;\theta)=\alpha$. In some cases, this seems to depend on the tail behavior of the baseline (i.e., the untruncated) distribution. Let $Y$ be a random variable originally coming from the location family with CDFs $F_0(y-\theta)$, $\theta\in\mathbb R$. Assume that $Y$ is (possibly doubly) truncated in the interval $(T_1,T_2)$, where $T_1$ or $T_2$ can be $-\infty$ and $+\infty$, respectively. Then, the CDF of $Y$ is $F(y;\theta)=\{F_0(y-\theta)-F_0(T_1-\theta)\}/\{F_0(T_2-\theta)-F_0(T_1-\theta)\}$, $y\in(T_1,T_2)$. It can be verified that for any fixed $y$,
\begin{align}
\lim_{\theta\uparrow\infty}F(y;\theta)=~& \left\{\begin{array}{ll}\displaystyle\lim_{\theta\downarrow-\infty}\frac{F_0(y-T_2+\theta)}{F_0(\theta)}, & \mbox{if $T_2<\infty$}, \\ 0, & \mbox{otherwise},\end{array}\right. \label{left tail}
\\[1ex]
\lim_{\theta\downarrow-\infty}F(y;\theta)=~&
\left\{\begin{array}{ll}\displaystyle\lim_{\theta\uparrow\infty}\frac{\bar{F}_0(y-T_1+\theta)}{\bar{F}_0(\theta)}, & \mbox{if $T_1>-\infty$}, \\ 1, & \mbox{otherwise}, \end{array}\right. \label{right tail}
\end{align}
where $\bar{F}_0=1-F_0$. Assume further that $F(y;\theta)$ is decreasing in $\theta$. For instance, this is ensured if the baseline distribution has the monotone likelihood ratio property with respect to $\theta$.
Then, the existence of $\theta(\alpha,y)$ depends on the behavior of the tails of $F_0$. For example, when $F_0(y)=\Phi(y)$, $y\in\mathbb R$, i.e., the standard normal CDF, the above limits equal $0$ and $1$, respectively, for any $y$ (and $T_1<T_2$). Thus, $\theta(\alpha,y)$ always exists. On the other hand, when $F_0(y)=e^y/(1+e^y)$, i.e., the standard logistic CDF, the limits are $e^{y-T_2}$ and $e^{T_1-y}$, respectively, which means that for any particular $y$, $\theta(\alpha,y)$ exists only for a restricted range of $\alpha$'s.

\section*{References}

\begin{description}\setlength{\itemsep}{-.5ex}
\item Arnold, B.C., Balakrishnan, N., Nagaraja, H.N.~(2008). {\it A First Course in Order Statistics}, Classic Edition, SIAM, Philadelphia.
\item Balakrishnan, N.~(2007). Progressive censoring: an appraisal (with discussions), {\it Test}, 16, 211--296.
\item Balakrishnan, N., Brain, C., Mi, J.~(2002). Stochastic order and mle of the mean of the exponential distribution, {\it Methodology and Computing in Applied Probability}, 4, 83--93.
\item Balakrishnan, N., Han, D., Iliopoulos, G.~(2011). Exact inference for progressively type-I censored exponential failure data, {\it Metrika}, 73, 335--358.
\item Balakrishnan, N., Iliopoulos, G.~(2009). Stochastic monotonicity of the MLE of exponential mean under different censoring schemes, {\it Annals of the Institute of Statistical Mathematics}, 61, 753--772.
\item Balakrishnan, N., Kundu, D.~(2013). Hybrid censoring: Models, inferential results and applications (with discussions), {\it Computational Statistics and Data Analysis}, 57, 166--209.
\item Barlow R.E., Madansky A., Proschan F., Scheuer E.~(1968). Statistical estimation procedures for the ``burn-in'' process. {\it Technometrics}, 10, 51--62.
\item Bartholomew, D.J.~(1963). The sampling distribution of an estimate arising in life testing, {\it Technometrics}, 5, 361--374.
\item Bartlett, M.S.~(1953). Approximate confidence intervals, {\it Biometrika}, 40, 12--19.
\item Casella, G. and Berger, R.L.~(2002). {\it Statistical Inference}, Second edition, Duxbury Press, Boston.
\item Chandrasekar, B., Childs, A., Balakrishnan, N.~(2004). Exact likelihood inference for the exponential distribution under generalized Type-I and Type-II hybrid censoring, {\it Naval Research Logistics}, 7, 994--1004.
\item Charalambides, C.A.~(2005). {\it Combinatorial Methods in Discrete Distributions}, Wiley, Chichester.
\item Chen, S., Bhattacharyya, G.K.~(1988). Exact confidence bounds for an exponential parameter under hybrid censoring, {\it Communications in Statistics -- Theory and Methods}, 17, 1857--1870.
\item Childs, A., Chandrasekar, B., Balakrisnan N., Kundu, D.~(2003). Exact likelihood inference based on Type-I and Type-II hybrid censored samples from the exponential distribution, {\it Annals of the Institute of Statistical Mathematics}, 55, 319--330.
\item Childs, A., Chandrasekar, B., Balakrishnan, N.~(2008). Exact likelihood inference for an exponential parameter under progressive hybrid censoring schemes, In: 
{\it Statistical Models and Methods for Biomedical and Technical Systems} (Eds., F. Vonta, M. Nikulin, N. Limnios and C. Huber-Carol), pp. 323--334, Birkh\"auser, Boston.
\item Clopper, C.J., Pearson, E.S.~(1934). The use of confidence or fiducial limits illustrated in the case of the binomial, {\it Biometrika}, 26, 404--413.
\item Cramer, E., Balakrishnan, N.~(2013). On some exact distributional results based on type-I progressively hybrid censored data from exponential distributions, {\it Statistical Methodology}, 10, 128--150.
\item Epstein, B.~(1954). Truncated life tests in the exponential case, {\it Annals of Mathematical Statistics}, 25, 555--564.
\end{description}

\end{document}